\documentclass[12pt]{article}
\usepackage{amsfonts}
\usepackage{amssymb}
\usepackage{latexsym}
\usepackage{subfigure}
\usepackage{amsmath}
\usepackage{amsthm}
\usepackage{fullpage}

\newtheorem{thm}{Theorem}[section]
\newtheorem{exa}[thm]{Example}
\newtheorem{lem}[thm]{Lemma}

\newtheorem{conj}[thm]{Conjecture}
\newtheorem{prob}[thm]{Problem}

\newcommand{\Z}{\mathbb{Z}}
\newcommand{\F}{\mathbb{F}}
\newcommand{\N}{\mathbb{N}}
\newcommand{\len}{\mathop{\rm len}\nolimits}
\makeatletter
\def\imod#1{\allowbreak\mkern10mu({\operator@font mod}\,\,#1)}
\makeatother

\begin{document}

\title{Distinct Partial Sums in Cyclic Groups: Polynomial Method and Constructive Approaches}

\author{Jacob~Hicks$^{1}$, M.~A.~Ollis$^{2}$ and John.~R.~Schmitt$^3$  \\
              \\
              {\it ${}^1$Department of Mathematics, University of Georgia,} \\    
              {\it  Athens, GA 30602, USA.} 
             \\
             \\
              {\it ${}^2$Marlboro College, P.O.~Box A, Marlboro,} \\    
              {\it Vermont 05344, USA.}             
              \\
              \\
              {\it ${}^3$Department of Mathematics, Middlebury College,} \\
               {\it Middlebury, VT 05753, USA. }  }


\maketitle

\begin{abstract}
Let~$(G,+)$ be an abelian group and consider a subset~$A \subseteq G$ with~$|A|=k$.  Given an ordering~$(a_1, \ldots, a_k)$ of the elements of~$A$, define its {\em partial sums} by~$s_0 = 0$ and $s_j = \sum_{i=1}^j a_i$ for~$1 \leq j \leq k$.
We consider the following conjecture of Alspach:  For any cyclic group~$\Z_n$ and any subset~$A  \subseteq \Z_n \setminus \{0\}$ with~$s_k \neq 0$,  it is possible to find an ordering of the elements of~$A$ such that no two of its partial sums~$s_i$ and~$s_j$ are equal for~$0 \leq i < j \leq k$.  We show that Alspach's Conjecture holds for prime~$n$  when~$k \geq n-3$ and when~$k \leq 10$.  The former result is by direct construction, the latter is non-constructive and uses the polynomial method.  We also use the polynomial method to show that for prime~$n$ a sequence of length $k$ having distinct partial sums exists in any subset of $\Z_n \setminus \{0\}$ of size at least $2k- \sqrt{8k}$ in all but at most a bounded number of cases.
\end{abstract}

\section{Introduction}\label{sec:intro}

We consider some simply stated problems and conjectures arising from the study of combinatorial designs.  These problems may be broadly described as follows: a finite subset of the elements of some group is given and one wishes to order the elements of this finite subset so that the sequence of partial sums has terms that are distinct.  Within this setting there are a plethora of questions that one might consider; these arise as one varies the group, places restrictions on the elements of the subset chosen, or imposes additional restrictions upon the sequence of partial sums beyond the terms being distinct.  It is surprising to us how recent these questions are and how relatively few results have been obtained thus far.

To be specific, let~$(G,+)$ be an abelian group and consider a subset~$A \subseteq G$ with~$|A|=k$.  Given an ordering~$(a_1, \ldots, a_k)$ of the elements of~$A$, define its {\em partial sums} by~$s_0 = 0$ and $s_j = \sum_{i=1}^j a_i$ for~$1 \leq j \leq k$.  

We will consider the following conjecture attributed to Alspach.

\begin{conj}\label{conj:alspach} {\rm (Alspach, see \cite{BH05})}
For any cyclic group~$\Z_n$ and any subset~$A  \subseteq \Z_n \setminus \{0\}$ with~$s_k \neq 0$,  it is possible to find an ordering of the elements of~$A$ such that no two of its partial sums~$s_i$ and~$s_j$ are equal for~$0 \leq i < j \leq k$.
\end{conj}

Bode and Harborth \cite{BH05} were the first to make a contribution when they established that Conjecture \ref{conj:alspach} is true whenever $|A|=n-1, n-2$.  They claimed (without proof) that the conjecture holds for $|A| \leq 5$ and verified it by computer for $n \leq 16$.  They stated that Alspach was motivated by the existence of cycle decompositions of complete graphs and complete graphs plus or minus a 1-factor (see \cite{AG01},\cite{S02}, and \cite{S03}), and of directed cycle decompositions of complete symmetric digraphs, \cite{AGSV03}.

Independent interest in Conjecture \ref{conj:alspach} arose via the work of Archdeacon \cite{A15}, who constructed embeddings of complete graphs so that the faces are $2$-colorable and each color class is a $k$-cycle system.  A confirmation of Conjecture \ref{conj:alspach} would have implications on the work in \cite{A15}.  That said, a weaker version of the conjecture was posited by Archdeacon, Dinitz, Mattern and Stinson in \cite{ADMS16}.

\begin{conj}\label{conj:archdeacon} {\rm (Archdeacon, Dinitz, Mattern and Stinson \cite{ADMS16})}
For any cyclic group~$\Z_n$ and any subset~$A \subseteq \Z_n \setminus \{0\}$,  it is possible to find an ordering of the elements of~$A$ such that no two of its partial sums~$s_i$ and~$s_j$ are equal for~$1 \leq i < j \leq k$.
\end{conj}

Note the distinctions between these two conjectures: Conjecture \ref{conj:alspach} stipulates that the sum of the elements of $A$ cannot be zero, while Conjecture \ref{conj:archdeacon} does not; and, Conjecture \ref{conj:alspach} essentially forbids any of the ``proper" partial sums from being zero since it considers $s_0$ (which equals zero), while Conjecture \ref{conj:archdeacon} does not.  Never-the-less, among their results, Archdeacon et al. \cite{ADMS16} proved that Conjecture \ref{conj:alspach} implies Conjecture \ref{conj:archdeacon}.  Furthermore, they verified via computer that Conjecture \ref{conj:archdeacon} is true for $n \leq 25$ via a ``guess and check" strategy.  They also proved that Conjecture \ref{conj:archdeacon} is true for $|A| \leq 6$.  

If Conjecture \ref{conj:archdeacon} is not true, then we might consider the following question, which was posed in~\cite{ADMS16}.

\begin{prob}\label{prob:sub} For any cyclic group~$\Z_n$ and any positive integer $k$, what is the smallest order such that from all subsets~$A \subseteq \Z_n \setminus \{0\}$ of that order we can construct a sequence of distinct elements of length $k$ that has distinct partial sums?
\end{prob}



Of course, if the length of the longest such sequence is $|A|$, then Conjecture \ref{conj:archdeacon} is valid.  

Another related conjecture was recently proposed by Costa, Morini, Pasotti and Pellegrini \cite{CMPP18}.

\begin{conj}\label{conj:3}{\rm (Costa, Morini, Pasotti, and Pellegrini \cite{CMPP18})}
For any abelian group~$(G,+)$ and any subset~$A \subseteq G \setminus \{0\}$ such that there is no~$x \in A$ with $\{ x, -x\} \subseteq A$ and with $s_k=0$,  it is possible to find an ordering of the elements of~$A$ such that no two of its partial sums~$s_i$ and~$s_j$ are equal for~$1 \leq i < j \leq k$.
\end{conj}

As pointed out by Costa et al. \cite{CMPP18}, in the case of $G= \Z_n$, Conjecture \ref{conj:3} immediately follows from Conjecture \ref{conj:archdeacon}.  If one replaces $\Z_n$ by an arbitrary abelian group in the statement of Conjecture \ref{conj:archdeacon}, Costa et al. \cite{CMPP18} established its validity when the order of the group is at most 23 and pointed out that the work of Archdeacon et al. confirming Conjecture \ref{conj:archdeacon} for sets $A$ of size at most 6 extends to abelian groups as well.

Costa et al. \cite{CMPP18} proved Conjecture \ref{conj:3} when $|A| \leq 9$ and verified by computer that it holds for abelian groups of order at most 27.  They stated that the conjecture arose during the study of Heffter systems (see \cite{CMPP18}, Section 2).

One of the seeming difficulties of these types of problems was discussed in \cite{ADMS16}.  For a fixed group of order $n$, there are $2^n-1$ non-empty subsets of $G$ (respectively, $2^{n-1}-1$ of $G \setminus{\{0\}}$) and so there are many problems to be solved for each $n$.  Those authors point out that the lack of structure in general of these subsets is an obstacle.  Indeed, the proofs for subsets of small size for these conjectures to be found in \cite{ADMS16} and \cite{CMPP18} illustrate this way of thinking.  In proving Conjecture \ref{conj:archdeacon} for just the case $|A|=6$, Archdeacon et al. \cite{ADMS16} break the proof into $\frac{|A|}{2}+1$ cases depending upon the number of pairs $\{x, -x\}$ contained in $A$ and each of these cases breaks into between three and nine subcases.  There is a similarly large amount of casework done in \cite{CMPP18} for Conjecture \ref{conj:3} for $|A|=9$.  In each paper, the proofs are constructive.





Our main contribution to these conjectures and problems is to show how the polynomial method (in the form of Noga Alon's Combinatorial Nullstellensatz \cite{A99}) may be applied.  In doing so, we {\em find structure} within the encoding polynomials, exploiting this to prove the validity of certain cases of these conjectures.  Our proofs using this method are non-constructive. 

In the case that $n$ is a prime, let us now show how to turn these combinatorial problems into algebraic ones.  We do so by constructing polynomials in which the non-zeros of a given polynomial correspond to solutions (and zeros to ``non-solutions") of the respective problem or conjecture.

For each of the above conjectures and the problem, we seek an ordering of the elements of $A$.  Let us associate to the $i^{th}$-entry of an ordering of length $k$ a variable $x_i$.  With these $k$ variables -- for each conjecture and the problem -- we construct a polynomial over a finite field of order $p$.  We denote this field using the ring notation $\Z_p$.  The set $A$ will serve as the set of inputs for each of the $k$ variables and so each of the polynomials will be defined over $A \times \ldots \times A = A^k$.  For each conjecture and the problem, we seek the ordering to be a permutation of the elements of $A$, we desire that $x_i \neq x_j$ for $1 \leq i < j \leq k$, i.e. $x_i -x_j \neq 0$ for $1 \leq i < j \leq k$.

With respect to Conjecture \ref{conj:alspach}, in addition to the above, we seek the ordering to have no two of its partial sums equal for ~$0 \leq i < j \leq k$, which is to say that we desire $\sum_{\ell=1}^j x_{\ell} \neq 0$ for $1 \leq j \leq k$ (as the empty partial sum equals $0$) {\em and} $\sum_{\ell=1}^{i} x_{\ell} \neq \sum_{\ell=1}^j x_{\ell}$ for $1 \leq i < j \leq k$.  As the hypothesis of Conjecture \ref{conj:alspach} gives that the sum of all the elements of $A$ is non-zero, we may drop the requirement that $\sum_{\ell=1}^k x_{\ell} \neq 0$.  With a re-indexing the second set of inequalities may be re-expressed as $\sum_{\ell=1}^{j} x_{\ell} - \sum_{\ell=1}^{i-1} x_{\ell} \neq 0$ for $1 \leq i-1 < j \leq k$ or rather more simply $x_i+ \ldots + x_j \neq 0$ for $2 \leq i < j \leq k$.  As we seek to satisfy each of these three sets of linear constraints {\em simultaneously}, we consider the product as given in the following polynomial, which belongs to the polynomial ring $\Z_p[x_1, \ldots ,x_k]$.


$$F_k:=F_k(x_1, \ldots, x_k) =  \prod_{1 \leq i < j \leq k} (x_j - x_i)   (x_i + \cdots + x_{j})  \ \ /  \ \ (x_1 + \cdots + x_k) $$

It should now be apparent that the inputs from $A^k$ that output a non-zero value in $F_k$ are solutions to Conjecture \ref{conj:alspach}.

With respect to Conjecture \ref{conj:archdeacon}, in addition to requiring that $x_i -x_j \neq 0$ for $1 \leq i < j \leq k$, we seek the ordering to have no two of its partial sums equal for ~$1 \leq i < j \leq k$, which is to say that we desire $\sum_{\ell=1}^{i} x_{\ell} \neq \sum_{\ell=1}^j x_{\ell}$ for $1 \leq i < j \leq k$.  With a re-indexing this set of inequalities may be re-expressed as $\sum_{\ell=1}^{j} x_{\ell} - \sum_{\ell=1}^{i-1} x_{\ell} \neq 0$ for $1 \leq i-1 < j \leq k$ or rather more simply $x_i+ \ldots + x_j \neq 0$ for $2 \leq i < j \leq k$.  As we seek to satisfy each of these two sets of linear constraints {\em simultaneously}, we consider the product of these as given in the following polynomial, which belongs to the polynomial ring $\Z_p[x_1, \ldots ,x_k]$.

$$f_k:= f_k(x_1, \ldots, x_k) = \prod_{1 \leq i < j \leq k} (x_j - x_i)  \prod_{2 \leq i < j \leq k}  (x_i + \cdots + x_{j})$$

The inputs from $A^k$ that output a non-zero value in $f_k$ are solutions to Conjecture \ref{conj:archdeacon}.

We note that, 
$$ F_k(x_1, \ldots , x_k) =  f_k(x_1, \ldots , x_k) \cdot \prod_{2 \leq j \leq k-1} (x_1 + \cdots + x_j). $$


With respect to Conjecture \ref{conj:3}, the requirements are the same as those of Conjecture \ref{conj:archdeacon}.  However, the additional hypothesis that there is no~$x \in A$ with $\{ x, -x\} \subseteq A$ allows us to drop from consideration constraints of the following form: $x_i+x_{i+1} \neq 0$.  Thus, we consider the following polynomial, where the inputs from $A^k$ that yield non-zero outputs are solutions to Conjecture \ref{conj:3}.

$$f_k(x_1, \ldots , x_k) / \left( \prod_{1 \leq i < k} (x_{i} + x_{i+1}) \right)= \prod_{1 \leq i < j \leq k} (x_j - x_i)  \prod_{\substack{2 \leq i < j \leq k \\ j \neq i+1}}  (x_i + \cdots + x_{j})$$

In this paper, we address these conjectures and the problem, mostly focusing on the case when $n$ is {\em prime} and do the following.  In Section \ref{sec:CN} we state the Combinatorial Nullstellensatz and show how it can be used to solve these conjectures.  Further, in the case of ${\mathbb Z}_n$ where $n$ is prime, we verify computationally that Conjecture \ref{conj:alspach} is true for $|A| \leq 10$ (and so Conjecture \ref{conj:archdeacon} and Conjecture \ref{conj:3} also hold for $|A| \leq 10$).  In Section \ref{sec:seqs}, in the case that $n$ is prime, we use the Combinatorial Nullstellensatz to address Problem \ref{prob:sub}, showing that a sequence of length $k$ exists in any subset of $\Z_n \setminus \{0\}$ of size at least $2k- \sqrt{8k}$ in all but at most a bounded number of cases.  In Section \ref{sec:constr} we show how the constructive method of Bode and Harborth \cite{BH05} may be extended to verify Conjecture \ref{conj:alspach} whenever $|A|=n-3$ and $n$ is prime.  In Section \ref{sec:conclude} we discuss a generalization of the Combinatorial Nullstellensatz that applies to the case when $n$ is composite so long as no two distinct elements of $A$ differ by a zero-divisor.

 





\section{The polynomial method approach}\label{sec:CN}

As seen in Section \ref{sec:intro} the inputs from $A^k$ that correspond with nonzero outputs of $F_k$ and $f_k$ are solutions to Conjecture \ref{conj:alspach} and Conjecture \ref{conj:archdeacon}, respectively.  We will use the following theorem due to Alon \cite{A99} to show that such nonzero outputs exist.

\begin{thm}\label{th:pm}{\rm (Alon's Non-vanishing Corollary, \cite{A99})}
Let ~$\F$ be an arbitrary field, and let $f = f(x_1, \ldots, x_k)$ be a polynomial in~$\F[x_1, \ldots, x_k]$.  Suppose the degree~$deg(f)$ of~$f$ is $\sum_{i=1}^k t_i$, where each~$t_i$ is a nonnegative integer, and suppose the coefficient of~$\prod_{i=1}^k x_i^{t_i}$ in~$f$ is nonzero.  Then if~$A_1,\ldots, A_k$ are subsets of~$F$ with~$|A_i| > t_i$, there are $a_1 \in A_1, \ldots, a_k \in A_k$ so that~$f(a_1, \ldots, a_k) \neq 0$.
\end{thm}

In the seminal work entitled {\it Combinatorial Nullstellensatz}, Alon \cite{A99} showed how Theorem \ref{th:pm} may be used to elegantly and quickly prove numerous statements from combinatorial number theory, combinatorics and graph theory.  It has been used dozens of times since then.

To apply Theorem \ref{th:pm} to the polynomial $F_k$, a polynomial of degree $2{k \choose 2}-1=k(k-1)-1$, we must identify a monomial of degree $k(k-1)-1$ for which the degree of each $x_i$ factor is less than $|A|=k$ and for which the coefficient of this monomial is nonzero.  There are $k$ such monomials to consider and these have the form $x_1^{k-1}\cdots x_j^{k-2}\cdots x_k^{k-1}$ for $1 \leq j \leq k$, that is for $1 \leq j \leq k$ the variable $x_j$ has degree $k-2$ and the other $k-1$ variables each have degree $k-1$.  Let $m_{k,j}$ denote such a monomial.

As opposed to computing the coefficients of these monomials over $\Z_p$ for each possible prime $p$, we compute the coefficients over $\Z$ for the following reason.  Consider a nonzero coefficient $c_{k,j}$ of some monomial $m_{k,j}$ of the polynomial $F_k$ over the ring $\Z$ and suppose $c_{k,j} \neq 0$.  Let the coefficient $c_{k,j}$ have prime factorization $p_1^{e_1}\cdots p_b^{e_b}$.  Thus, the coefficient of the monomial $m_{k,j}$ over $\Z_p$ is not zero if and only if $p \not \in {\mathbb P}_{k,j}:=\{p_1, \ldots , p_b\}$.  As $k \leq p-1$ primes that are present in the factorization and that do not obey this inequality need not be considered.  This means that Conjecture \ref{conj:alspach} holds for that value of $k$ and all primes $p$ not in this list, i.e. for $p > k$ and $p \not \in \{p_1, \ldots , p_b\}$.  We may then turn to a different nonzero coefficient, say, $c_{k,i}$ for $i \neq j$ and repeat this argument. This means that if $\cap_{i=1}^k {\mathbb P}_{k,i}$ is empty or only contains primes less than $k$, then Conjecture \ref{conj:alspach} holds for that value of $k$ and all primes $p$.

The results of these computations for $2 \leq k \leq 10$ are given in Table \ref{array:alspachcoefficients}.  Due to symmetries within $F_k$, we have that $c_{k,j}=\pm c_{k,k-j}$ for all $k$ and $1 \leq j \leq \lfloor \frac{k}{2} \rfloor$.  The computations used magma \cite{BCP97} for polynomial multiplication along with several optimizations for the specific problem. The optimizations included checking the degree of particular variables and if they grew too large to be of the desired form the algorithm would eliminate all monomials containing it from future computations. A similar method was employed if a particular variable's degree was too small. The multiplication was carried out term-by-term and combined at the end of each step to minimize the search time for the optimizations.  All of the computations were completed in less then 3 hours on a MacBook Pro with a 3.1 GHz Intel Core i5 processor and 16 GB of RAM.


\begin{table}
\caption{Values of $c_{k,j}$}\label{array:alspachcoefficients}
$$
\begin{array}{c|ccccc}
k \setminus j & 1 & 2 & 3 & 4 & 5 \\
\hline
2 & 1 \\
3  &  -1 &  0  \\
4  &  1 & -1     \\
5  &  4 &  -2 & -4   \\
6  &  -28 & -40 & -20  \\
7  &  966 & 1662 & 1338 & 0  \\
8 & -366468 & -92412 & 144324 & 314556  \\
9 & -359616276 & -130597656 & 72122706 &  254703096 & 326776260 \\
10 & 595372941856 & 1404671795722 & 1785841044600 & 1435120776421 & 546395688803  \\
\end{array}
$$
\end{table}


\begin{thm}\label{th:alspach}
Alspach's Conjecture (Conjecture~\ref{conj:alspach}) is true for prime~$n$ and $k \leq 10$.
\end{thm}

\begin{proof}
For each $k$, consider the set of $c_{k,j}$ given in Table~\ref{array:alspachcoefficients}.  The set $\cap_{j=1}^k {\mathbb P}_{k,j}$ is either empty or only contains primes less than $k$.

\end{proof}

\begin{thm}\label{th:archdeacon}
Archdeaon,  Dinitz, Mattern and Stinson's Conjecture (Conjecture~\ref{conj:archdeacon}) and Costa, Morini, Pasotti and Pellegrini's Conjecture (Conjecture~\ref{conj:3}) are true for prime~$n$ and $k \leq 10$. 
\end{thm}

\begin{proof}
The result follows from Theorem~\ref{th:alspach}.  
\end{proof}





\section{Sequences taken from given subsets}\label{sec:seqs}

In this section we consider Problem~\ref{prob:sub} in the case that~$n=p$ is prime.



The polynomial~$f_k$ is the one of interest:
$$f_k(x_1, \ldots, x_k) = \prod_{1 \leq i < j \leq k} (x_j - x_i)  \prod_{2 \leq i < j \leq k}  (x_i + \cdots + x_{j}).$$
Recall that it is homogeneous and has degree~$(k-1)^2$.  As we saw earlier,  a monomial of~$f_k$ with nonzero coefficient and highest exponent~$k-1$ proves Conjecture~\ref{conj:archdeacon} for that value of~$k$ and all but finitely many prime values of~$p$.  More generally, the same argument shows that if we can find a leading monomial with nonzero coefficient and degree of the highest term~$2k-d-1$ then Alon's Non-vansishing Corollary implies a solution to Problem~\ref{prob:sub} when~$|A| = 2k-d$ (again, for all but finitely many prime values of~$p$).

Two other homogeneous polynomials will be useful.  
$$g_k(x_1, \ldots, x_k) = \prod_{1 \leq i < j \leq k} (x_j - x_i)  (x_i + \cdots + x_{j})$$
$$h_k (x_1, \ldots, x_k) = \prod_{1 \leq i <k} (x_k - x_i)  (x_i + \cdots + x_k)$$
Note that~$g_k = g_{k-1}h_k  = h_2h_3 \cdots h_{k-1}h_k$.  (When the variables of a polynomial are not specified, we take them to run from~$x_1$ upwards in sequence.  So, for example, $h_{k-1} = h_{k-1}(x_1, \ldots, x_{k-1})$.)

As an illustration of the method, we first find the coefficient on
$$x_1^{k-1}x_2^0x_3^2x_4^4 \cdots x_k^{2k-4}$$
in $f_k$.  Note that the degree of this monomial is $$(k-1) + \sum_{i=1}^{k-2} 2i = (k-1)^2 = \deg(g_k).$$
We have
$$f_k = \prod_{1 < j  \leq k} (x_j - x_1)  g_k(x_2, x_3, \ldots, x_k).$$
There are only~$k-1$ factors that include~$x_1$ in~$f_k$.  Hence the coefficients on any monomials in~$f_k$ that include $x_1^{k-1}$ are given by the coefficients on the same monomials in
$$ (-1)^{k-1}x_1^{k-1}g_{k-1}(x_2, x_3, \ldots, x_k).$$

We can therefore turn our attention to the~$g$ polynomials and we want to find the coefficient on $x_2^0x_3^2 \cdots x_k^{2k-4}$ in~$g_{k-1}(x_2, x_3, \ldots, x_k)$.  After reindexing to make the notation smoother, we are interested in the coefficient on $x_1^0x_2^2 \cdots x_k^{2k-2}$ in~$g_k$.
Now, $g_k =  h_2h_3 \cdots h_{k-1}h_k$ and we consider the contributions of the~$h_j$ polynomials in turn, starting with~$h_k$ and working downwards.  The term~$x_k$ only appears in~$h_k$, appears (with a~$+$~sign) in every factor of~$h_k$, and~$\deg(h_k) = 2k-2$.  Therefore, when expanding the polynomial, we must select all of these~$x_k$ terms.  Now $h_{k-1}$ has degree~$2k-4$ and is the only remaining~$h_j$ to feature~$x_{k-1}$.  Similarly to the~$h_k$ case, we get a contribution of~$x_{k-1}^{2k-4}$ here.  Continuing, we see that~$g_k$ has a coefficient of~$+1$ on~$x_1^0x_2^2 \cdots x_k^{2k-2}$ and hence $f_k$ has a coefficient of~$(-1)^{k-1}$ on $x_1^{k-1}x_2^0x_3^2x_4^4 \cdots x_k^{2k-4}$.

In~\cite{CMPP18}, a greedy algorithm approach to Problem~\ref{prob:sub} is used.  The preceding discussion allows us to slightly improve this result in the case where~$n$ is prime:

\begin{thm}\label{th:warm_up}
Let~$p$ be prime.  From any subset of size~$2k-3$ of $\Z_p \setminus \{ 0 \}$ we can construct a sequence of length~$k$ with distinct partial sums.
\end{thm}

\begin{proof}
As shown in the preceding discussion, $x_1^{k-1}x_2^0x_3^2x_4^4 \cdots x_k^{2k-4}$ is a monomial in~$f_k$ with coefficient~$\pm 1$.  As this monomial has coefficient coprime to~$p$ and highest exponent at most~$2k-4$, the result follows from Theorem~\ref{th:pm}.
\end{proof}

We now reach the main result of this section.  It says that if~$p$ is an odd prime then for~$k \geq 8$ we can almost always find a sequence of~$k$ elements with distinct partial sums from any set of size at least~$2k-\sqrt{8k}$ in~$\Z_p \setminus \{ 0 \}$.  

As with Theorem~\ref{th:warm_up}, finding a nonzero coefficient on a monomial of~$f_k$ is at the heart of the proof.  The monomial we choose this time is:
$$x_1^{k-1}x_2^0x_3^2x_4^4 \cdots x_{k-d+2}^{2k-2d} x_{k-d+3}^{2k-d-1} x_{k-d+4}^{2k-d-1} \cdots x_{k}^{2k-d-1}.$$
As before, we can deal with the $x_1^{k-1}$ part separately and then consider the remaining monomial in $g_{k-1}(x_2, \ldots, x_k)$.  The highest exponents are crowded into the terms with the highest indices and this lets us use a similar breakdown of~$h_k$ to track the coefficient.

We need the notion of a polynomial in~$k$ of degree~$i$ and positive leading coefficient.  We use the notation~$\Theta(k^i)$ for this, and write, for example, $\Theta(k^2)x_{k-1}^{2k-6}$ to mean that the coefficient on the monomial $x_{k-1}^{2k-6}$ is quadratic in~$k$ with positive coefficient on the~$k^2$.  

\begin{thm}\label{th:2k-d}
Fix $d \in \mathbb{N}$ with~$d>3$ and let~$k > \min{(d-1, d^2/8)}$.  Then for all but at most $(d-3)(d-2)(d-1)/6$ of these values of~$k$ there is a monomial in~$g_k$ with a nonzero coefficient and largest exponent~$2k-d-1$.  Hence for almost all primes~$p$ there are at most $(d-3)(d-2)(d-1)/6$ values of~$k$ (with~$k > \min{(d-1, d^2/8)}$) where it is not the case that we can construct a sequence of~$k$ elements with distinct partial sums from any set of~$2k-d$ distinct elements of~$\Z_p \setminus \{ 0 \}$.
\end{thm}

\begin{proof}
Consider the monomials of~$f_k$ with exponent~$k-1$ on~$x_1$.  There are only~$k-1$ factors including~$x_1$ in~$f_k$.  Selecting these when expanding~$f_k$ gives
$$ (-1)^{k-1}x_1^{k-1}g_{k-1}(x_2, x_3, \ldots, x_k).$$
We can therefore turn our attention to the~$g$ polynomials. It is sufficient to find a nonzero coefficient on a monomial of
$g_{k-1}(x_2, x_3, \ldots, x_k)$ with highest exponent $2k-d-1$, provided that $2k-d-1 \geq k-1$, which holds by the hypothesis~$k > d-1$.   Equivalently, and avoiding the need to keep specifying the variables, we look for a nonzero coefficient on a monomial of~$g_k$ whose highest exponent is~$2k-d+1$.  

The monomial we focus our attention on is
$$m_{k,d} = x_1^0x_2^2x_3^4 \cdots x_{k-d+2}^{2k-2d+2} x_{k-d+3}^{2k-d+1} x_{k-d+4}^{2k-d+1} \cdots x_{k}^{2k-d+1}.$$
This has degree
\begin{eqnarray*}
\sum_{i=1}^{k-d+2} 2(i-1)  +  \sum_{i=k-d+3}^{k} (2k-d+1) &=& (k-d+1)(k-d+1) + (d-2)(2k-d+1) \\
      &=& k(k-1) \\
      &=& \deg(g_k).
\end{eqnarray*}
Therefore this is a monomial of maximum degree of~$g_k$ (which is also implied by the homogeneity of~$g_k$).  (The corresponding monomial
$$ x_1^{k-1}x_2^0x_3^2x_4^4 \cdots x_{k-d+2}^{2k-2d} x_{k-d+3}^{2k-d-1} x_{k-d+4}^{2k-d-1} \cdots x_{k}^{2k-d-1}  $$
of degree $(k-1)^2$ is also maximal in~$f_k$.)

As noted earlier, $g_k = h_2h_3 \cdots h_{k-1}h_k$.  We have~$\deg(h_j) = 2j-2$.

Consider the variables in~$m_{k,d}$ with index at least~$k-d+3$ (that is, the ones with exponent~$2k-d+1$).  As~$x_i$ appears in~$h_j$ if and only if~$j \geq i$, these must be produced from the~$h_j$ with~$j \geq k-d+3$.  On the other hand, 
\begin{eqnarray*}
\deg(h_{k-d+3}h_{k-d+4} \cdots h_{k}) &=& \sum_{j=k-d+3}^{k} 2j-2 \\
                                                          &=& (d-2)(2k-d+1)  \\
                                                          &=& \deg(x_{k-d+3}^{2k-d+1} x_{k-d+4}^{2k-d+1} \cdots x_{k}^{2k-d+1})
\end{eqnarray*}
and so $h_{k-d+3}h_{k-d+4} \cdots h_{k}$ accounts for variables in~$m_{k,d}$ with index at least~$k-d+3$, and the entirety of their degree.

We decompose~$h_k$ further by writing $h_k = p_kq_k$ where
$$p_k (x_1, \ldots, x_k) = \prod_{1 \leq i <k} (x_k - x_i)$$
$$q_k (x_1, \ldots, x_k) = \prod_{1 \leq i <k}  (x_i + \cdots + x_k)$$
We have
$ \deg(p_j) = \deg(q_j) = j-1$.

For any polynomial~$p_k$, let~$p_{k|k-c}$ denote the polynomial obtained from~$p_k$ by removing all monomials with a variable whose index is less than~$k-c$.  

We now examine~$p_k$ and~$q_k$ in sufficient detail to get the result we require.  The easier one is~$p_k$:
$$p_k = x_k^{k-1} - \sum_{1\leq i < k} x_ix_k^{k-2} + \sum_{1 \leq i_1 < i_2 < k} x_{i_1}x_{i_2}x_k^{k-3} - \cdots $$
For~$q_k$ we restrict our attention to~$q_{k|k-d+3}$:
$$q_{k|k-d+3} =   x_k^{k-1} + \sum_{k-d+3 \leq i < k} i x_ix_k^{k-2} +
   \sum_{k-d+3 \leq i_1 < i_2 < k} \left( {i_1 \choose 2}x_{i_1}^2 + i_1(i_2 -1)x_{i_1}x_{i_2} \right) x_k^{k-3} + \cdots $$
In this case it's not quite so clear what the ``$\cdots$" implies.

For~$e \in \mathbb{N}$, let
$$V_e = \{(i_1, \ldots, i_e) : i_j \in \N,  \ 1 \leq i_1 \leq \cdots \leq i_e < k {\rm \ and \ } i_{\ell} \geq \ell {\rm \ for \ each \ } \ell  \}.$$  (Note that the inequalities present in the definition of $V_e$ imply that $e \leq k-1$.)
Let~$v \in V_e$ and~$\#_v(x)$ be the number of times~$x$ appears in a sequence~$v$. Then the terms in~$q_{k|k-d+3}$ with~$x_k$ raised to the power~$(k-1)-e$ are:
$$\sum_{v \in V_e} \frac{i_1(i_2-1)(i_3-2) \cdots (i_e- (e-1))}{\#_v(1)! \#_v(2)! \cdots \#_v(k-1)!}x_{i_1}x_{i_2}\cdots x_{i_e} x_k^{(k-1)-e}.$$
What's going on:  Choose the factors to give the~$x_{i_{\ell}}$ in non-descending order of index. First, the condition~$i_{\ell} \geq \ell$ ensures that there is at least one to pick at each step.  If~$i_{\ell}$ appears exactly once in~$v$ then we have~$i_{\ell} - (\ell - 1)$ choices for~$i_{\ell}$ at that stage. In general, when we come to select the factors to give~$x_{i_{\ell}}^{\#_v(i_{\ell})}$ for $\#_v(i_{\ell}) \geq 1$ we have
$${i_l - ({\ell} -1) \choose \#_v(i_{\ell})} = \frac{(i_{\ell} - (\ell -1 ))(i_{\ell} - \ell) \cdots (i_{\ell} - (\ell + \#_v(i_{\ell}) - 2))}{\#_v(i_{\ell})!}$$ 
ways to do so.

The coefficient on $x_{i_1}x_{i_2}\cdots x_{i_e} x_k^{(k-1)-e}$ is~$\Theta(k^{e})$ when~$(i_1, \ldots, i_e) \in V_e$ and~0 otherwise.    
Therefore
\begin{flushleft}
$q_{k|k-d+3} =   x_k^{k-1} + \sum_{k-d+3 \leq i_1 < k} a_{(i_1)} x_{i_1}x_k^{k-2} +
   \sum_{k-d+3 \leq i_1 \leq i_2 < k} a_{(i_1,i_2)} x_{i_1}x_{i_2}x_k^{k-3} +$
   \end{flushleft}
   \begin{flushright}
   $
    \sum_{k-d+3 \leq i_1 \leq i_2 \leq i_3 < k} a_{(i_1,i_2,i_3)} x_{i_1}x_{i_2}x_{i_3}x_k^{k-4} + \cdots $
    \end{flushright}
where each of the~$a_{v}$ is either~$\Theta(k^{\len{v}})$, where~$\len{v}$ is the length of~$v$, or~0, depending on whether~$v \in V_e$.

Recombining~$p_k$ and~$q_{k|k-d+3}$ we see that
$$h_{k|k-d+3} = x_k^{2k-2} + \sum_{k-d+3 \leq i_1 < k} b_{(i_1)} x_{i_1}x_k^{2k-3} +
  \sum_{k-d+3 \leq i_1 \leq i_2 < k} b_{(i_1,i_2)} x_{i_1}x_{i_2}x_k^{2k-4} + \cdots $$
where again each of the~$b_{v}$ is either~0 or~$\Theta(k^{\len{v}})$ (the negative signs in the~$p_k$ expansion do not affect the highest power of~$k$).  

The coefficient on~$x_{k-d+3}^{2k-d+1} x_{k-d+4}^{2k-d+1} \cdots x_{k}^{2k-d+1}$ in~$h_{k-d+3}h_{k-d+4} \cdots h_{k}$ is obtained by adding and multiplying coefficients on such monomials and so is a (possibly trivial) polynomial in~$k$.  Call this polynomial~$\alpha(k)$.  We claim that~$\alpha(k)$ is not trivial and that~$\deg(\alpha(k))$ is~$(d-3)(d-2)(d-1)/6$.

Consider the contributions of~$h_{k-d+3}, h_{k-d+4}, \ldots$ in turn.  The first does not fall into the general pattern: $h_{k-d+3}$ must contribute~$x_{k-d+3}^{2k-2d+4}$ to~$m_{k,d}$ since it is the only term of $h_{k-d+3|k-d+3}$.  This has coefficient~1.

Suppose now that we wish to find the degree of~$\alpha(k)$.  We can do so by maximizing the power of~$k$ in the contribution at each successive step.   (We shall see that doing so does not mean that we have foregone the possibility of a larger power of~$k$ later in the process.)

Looking at~$h_{k-d+4}$ we are concerned with monomials of the form~$x_{k-d+3}^c x_{k-d+4}^{2k-2d+6 - c}$.  We maximize the exponent of~$k$ on the coefficient by taking~$c$ as large as possible.  That is, $c = d-3$, giving coefficient in~$\Theta(k^{d-3})$ (or possibly 0, but we shall see that it is indeed $\Theta(k^{d-3})$ below).

Turning to~$h_{k-d+5}$ we are faced with a similar choice, and the pattern will continue as the index decreases.  In each case we again get a coefficient that must be $\Theta(k^i)$ or~0 for some~$i$.  We show that none of them are~0 in the next paragraph. 
Taking the largest choice of power on~$x_{k-d+4}$ in~$h_{k-d+5}$ gives the smallest on~$x_{k-d+5}$ and the coefficient is~$\Theta(k^{2d-8})$ or 0.  Continuing in this vein, the general step for~$h_{k-d+\ell}$ is to take the monomial~$x_{k-d+\ell-1}^{(\ell-3)(d-\ell+1)}x_{k-d+\ell}^{2k - (\ell-1)(d-\ell+1)}$.  This has coefficient~$\Theta(k^{(\ell-2)d - (\ell-1)^2 + 1})$ or 0.  The final step is to take the monomial~$x_{k-1}^{d-3}x_k^{2k-d+1}$ from~$h_k$ which has coefficient~$\Theta(k^{d-3})$ or 0.

Every choice gives us a nonzero coefficient.  To see this, note that we are looking at the variables with the two highest indices in~$h_{k-d+\ell}$; that is, $x_{k-d+\ell-1}$ and~$x_{k-d+\ell}$.  Each of these appears in every factor of both~$p_{k-d+\ell}$ and~$q_{k-d+\ell}$ and so the only possible impediment is that the exponent required on~$x_{k-d+\ell}$ is negative.  The expression~$(\ell-1)(d-\ell+1)$ is maximized at~$\ell=(d+2)/2$, giving a potential maximum value of~$d^2/4$.  We therefore require that~$d^2/4 < 2k$.  This is guaranteed by the hypothesis that~$k > d^2/8$.

These choices do indeed maximize the power of~$k$. Within the scheme for choosing which monomials to focus on, there is no point where we can make a choice that increases the power of~$k$.  If we make a choice  from~$h_{k-d+\ell}$ to increase the power on~$x_{k-d+\ell}$ (thus decreasing the power of~$k$ in the coefficient) then we must pick up the shortfall in instances of~$x_{k-d+\ell}$ elsewhere.  This can only happen in~$h_{x-d+\ell'}$ where~$\ell' > \ell$ and the need to  recover this shortfall  means that we cannot gain future higher choices for the power of~$k$ by doing this.

Now,
$$\deg(\alpha(k)) = \sum_{\ell = 4}^{d} (\ell-3)(d-\ell+1) = \frac{(d-3)(d-2)(d-1)}{6}$$
(for the summation formula see, for example, \cite[Sequence A000292]{Sloane}).


How does this relate to~$m_{k,d}$?  Our hand in choosing the rest of the contributing monomials to~$m_{k,d}$ from $h_2, h_3, \ldots, h_{k-d+2}$ is now forced.  We require the exponent on~$x_{k-d+2}$ to be~$2k-2d+2$.  From this list, this variable only appears in~$h_{k-d+2}$ which has degree~$2k-2d+2$ and so the monomial~$x_{k-d+2}^{2k-2d+2}$ is the only contributor to the total.  This pattern continues as we work backwards through~$h_2, h_3, \ldots, h_{k-d+2}$ meaning that the coefficient on~$m_{k,d}$ in~$g_k$ is the nonzero polynomial~$\alpha(k)$.

Finally, we translate this back to~$f_k$, the polynomial of primary interest.  The coefficient on 
$$x_2^0x_3^2x_4^4 \cdots x_{k-d+2}^{2k-2d} x_{k-d+3}^{2k-d-1} x_{k-d+4}^{2k-d-1} \cdots x_{k}^{2k-d-1}$$
in~$g_{k-1}(x_2,x_3,\ldots,x_k)$ is~$\alpha(k-1)$ and so the coefficient on 
$$x_1^{k-1}x_2^0x_3^2x_4^4 \cdots x_{k-d+2}^{2k-2d} x_{k-d+3}^{2k-d-1} x_{k-d+4}^{2k-d-1} \cdots x_{k}^{2k-d-1}$$
in~$f_k$ is $(-1)^{k-1}\alpha(k-1)$.

As the degree of the coefficient considered as a polynomial in~$k$ is~$(d-3)(d-2)(d-1)/6$, the remainder of the statement of the theorem follows.
\end{proof}

\section{The constructive approach}\label{sec:constr}

In this section we take a constructive approach to the problem when~$k$, the size of the subset we are looking for an ordering of, is close to~$n$, the size of the cyclic group from which the elements are taken.  We take~$n$ to be odd throughout the section and for the main part of the work~$n=p$ is an odd prime.  Our focus is on Alspach's Conjecture (Conjecture~\ref{conj:alspach}), the strongest of the three given in Section~\ref{sec:intro}.

When~$n$ is odd the sum of the nonzero elements of~$\Z_n$ is zero.  This observation tells us that the~$k=n-1$ case of Alspach's Conjecture is vacuously true for odd~$n$.  The~$k=n-2$ case follows quickly from the the existence of ``rotational sequencings" and our approach for the~$k=n-3$ case is also to exploit rotational sequencings, although a bit more work is required.

Let ${\bf a} = (a_1, a_2, \ldots, a_{n-1})$ be a cyclic arrangement of the nonzero elements of~$\Z_n$ (i.e.~$a_{n-1}$ is considered to be adjacent to $a_1$) and define~${\bf b} = (b_1, b_2, \ldots, b_{n-1})$ by $b_i = a_{i+1} - a_i$ where the indices are considered modulo~$(n-1)$ (so $b_{n-1} = a_1 - a_{n-1}$).     If the elements of~${\bf b}$ are distinct then ${\bf a}$ is a {\em directed rotational terrace} for~$\Z_n$ and~${\bf b}$ is its associated {\em  rotational sequencing}.  Clearly, the directed rotational terrace determines the  rotational sequencing; the reverse is also true.  (Note: there are several different but equivalent definitions in the literature and the less descriptive terms ``R-sequencing" and ``directed R-terraces" are often used instead; see, for example,~\cite{AAAP11,FGM78,K83,OW15}.)  

It is shown in \cite{FGM78} that~$\Z_n$ has a directed  rotational terrace if and only if $n$ is odd.  The method at the heart of their construction is one that we will repeatedly use and relies on graceful permutations (also known as graceful labelings of paths).

Let~$\boldsymbol{\alpha} = (\alpha_1, \alpha_2, \ldots, \alpha_r)$ be an arrangement of the first~$r$ positive integers and define~$\boldsymbol{\beta} = (\beta_1, \beta_2, \ldots, \beta_{r-1})$ by~$\beta_i = | \alpha_{i+1} - \alpha_i |$.  If the integers in~$\boldsymbol{\beta}$ are distinct, then~$\boldsymbol{\alpha}$ is a {\em graceful permutation} of length~$r$.   Call~$\boldsymbol{\beta}$ the sequence of {\em absolute differences} of~$\boldsymbol{\alpha}$.  Given a graceful permutation $\boldsymbol{\alpha} = (\alpha_1, \alpha_2, \ldots, \alpha_r)$, the sequence $(n+1 - \alpha_1, n+1 - \alpha_2, \ldots, n+1 - \alpha_r)$ is also a graceful permutation, called the {\em complement} of~$\boldsymbol{\alpha}$.

\begin{lem}\label{lem:fgm}{\rm \cite{FGM78} }
If $(\alpha_1, \alpha_2, \ldots, \alpha_r)$ is a graceful permutation then 
$$ (\alpha_1, \alpha_2, \ldots, \alpha_r, \alpha_r + r, \alpha_{r-1} + r, \ldots, \alpha_1 + r),$$
where the symbols are now considered as elements of~$\Z_{2r+1}$, is a directed  rotational terrace for~$\Z_{2r+1}$.
\end{lem}

\begin{exa}\label{ex:lww}
The sequence $(1,r,2,r-1,\ldots)$ is a graceful permutation of length~$r$ for each~$r$.  Call it the {\em Walecki Construction}~\cite{A08}. Therefore~$\Z_{2r+1}$ always has a rotational sequencing.
\end{exa} 

The above is all we need to prove Alspach's Conjecture for odd~$n$ and~$k=n-2$:

\begin{thm}\label{th:odd(n-2)}{\rm \cite{BH05} }
Let~$n$ be odd and take~$x \in \Z_n \setminus \{ 0 \}$.  Then the elements of~$\Z_n \setminus \{ 0,x \}$ can be ordered so that the partial sums are distinct and nonzero.
\end{thm}

\begin{proof}
Let ${\bf b} = (b_1, b_2, \ldots, b_{n-1})$ be a rotational sequencing of~$\Z_n$ with~$b_{n-1} = x$.  (This exists: take any  rotational sequencing and relabel the elements using the fact that the arrangement is cyclic.) We claim that $(b_1, b_2, \ldots, b_{n-2})$ is the required ordering.  Clearly it uses every nonzero element except~$x$.  The sums are all distinct and nonzero because otherwise there would be a repeat in the directed  rotational terrace associated with~${\bf b}$.
\end{proof}

Alspach's Conjecture for~$k=n-3$ may be approached in the same way:

\begin{lem}\label{lem:xyodd}
Let~$n$ be odd and let~$x$ and~$y$ be distinct nonzero elements of~$\Z_n$.  
If~$\Z_n$ has a rotational sequencing such that~$x$ and~$y$ are adjacent, then the elements of~$\Z_n \setminus \{ 0,x,y \}$ can be ordered so that the partial sums are distinct and nonzero.
\end{lem}

\begin{proof}
Reindex the rotational sequencing $(b_1, b_2, \ldots, b_{n-1})$ so that $\{b_{n-2}, b_{n-1} \} = \{ x, y \}$.  Then, just as in the proof of Theorem~\ref{th:odd(n-2)}, the ordering $(b_1, b_2, \ldots, b_{n-3})$ of $\Z_n \setminus \{ 0,x,y \}$ has distinct nonzero partial sums.
\end{proof}

For the remainder of the section, let~$n = p = 2r+1$ be prime.  Our first goal is to find a graceful permutations of length~$r$ with the properties we need to apply Lemmas~\ref{lem:fgm} and~\ref{lem:xyodd}.

As well as the Walecki Construction of Example~\ref{ex:lww} we need some of the ``twizzler" constructions developed in~\cite{OW11, P08}.  ``Twizzling" a sequence refers to a process of dividing it into subsequences and then reversing each of them.

The {\em 3-twizzler terrace} is a graceful permutation obtained by reversing successive 3-element subsequences of the Walecki construction (which we put between semi-colons to help parse the pattern) and then making sure the last few elements maintain the required properties.  There are six different cases as~$r$ varies modulo~$6$~\cite{P08}.

When~$r \equiv 0 \imod{6}$:
$$\left( 2, r, 1 ; r-2, 3, r-1 ;  5, r-3, 4 ; \ldots ;  \frac{r+2}{2}, \frac{r}{2}, \frac{r+4}{2}  \right).$$
When~$r \equiv 1 \imod{6}$:
$$\left( 2, r, 1 ; r-2, 3, r-1 ;  \ldots ; \frac{r-3}{2}, \frac{r+7}{2}, \frac{r-5}{2}; \frac{r+3}{2}, \frac{r+1}{2}, \frac{r+5}{2}, \frac{r-1}{2} \right).$$
When~$r \equiv 2 \imod{6}$:
$$\left( 2, r, 1 ; r-2, 3, r-1 ;  \ldots ; \frac{r+4}{2}, \frac{r-2}{2}, \frac{r+6}{2}; \frac{r+2}{2}, \frac{r}{2} \right).$$
When~$r \equiv 3 \imod{6}$:
$$\left( 2, r, 1 ; r-2, 3, r-1 ;  \ldots ;  \frac{r+1}{2}, \frac{r+3}{2}, \frac{r-1}{2}  \right).$$
When~$r \equiv 4 \imod{6}$:
$$\left( 2, r, 1 ; r-2, 3, r-1 ; \ldots ;   \frac{r+6}{2}, \frac{r-4}{2}, \frac{r+8}{2}; \frac{r}{2}, \frac{r+2}{2}, \frac{r-2}{2}, \frac{r+4}{2} \right).$$
When~$r \equiv 5 \imod{6}$:
$$\left( 2, r, 1 ; r-2, 3, r-1 ; \ldots ;  \frac{r-1}{2}, \frac{r+5}{2}, \frac{r-3}{2}; \frac{r+1}{2}, \frac{r+3}{2} \right).$$

Theorem~2 of \cite{OW11} constructs more general ``imperfect twizzler terraces" that are graceful permutations.  Here we extract the portion of the method that gives the permutations we require.   

Let~$\lceil r/3 \rceil < d < r$ and let~$(\gamma_1, \gamma_2, \ldots, \gamma_{d-1})$ be a graceful permutation of length~$d-1$ with~$\gamma_1 = d -  \lceil (r-d+1)/2 \rceil$, which exists as~$\gamma_1$ may be chosen to be any element in the range~$1 \leq \gamma_1 \leq d-1$
by a result proved independently in each of~\cite{C07,FFG83, G04}.
Then the sequence which starts with the reverse of the first~$r-d+1$ elements of the Walecki construction of length~$r$ followed by 
$$(\gamma_1 + \lceil (r-d+1)/2 \rceil, \gamma_2 + \lceil (r-d+1)/2 \rceil, \ldots, \gamma_{d-1} + \lceil (r-d+1)/2 \rceil)$$
is a graceful permutation of length~$r$.  Denote it~$\boldsymbol{\tau}_{d,r}$.

It will be useful to know what the possibilities for the first element of the absolute differences of a graceful permutation are.  Lemma~\ref{lem:diff1} completely characterizes this.

\begin{lem}\label{lem:diff1}
Take $1 \leq d < r$.   There is a graceful permutation of length~$r$ whose sequence of absolute differences starts with~$d$ if and only if~$(d,r) \not\in \{ (2,4), (2,5), (2,8) \}$.
\end{lem}

\begin{proof}
We require two constructions, one for small values of~$d$ and one for large ones.

The first construction is a slight generalization of one given in~\cite{AK90} and is also closely related to the imperfect twizzler construction.  Let~$\boldsymbol{\gamma} = (\gamma_1, \gamma_2, \ldots, \gamma_{\ell})$ be a graceful permutation. Let~$\boldsymbol{\alpha} = (\alpha_1, \alpha_2, \ldots, \alpha_{2t})$ be a graceful permutation of even length such that~$\alpha_{2i-1} > t$ for all $1 \leq i \leq t$ and that $\alpha_1 = \gamma_{\ell}+t$.    Then
$$(\gamma_1 + t, \gamma_2+t, \ldots, \gamma_{\ell}+t, \alpha_1+\ell, \alpha_2 , \alpha_3+\ell, \alpha_4, \ldots, \alpha_{2t-1}+\ell, \alpha_{2t})$$
is a graceful permutation of length~$\ell+2t$.  Such a graceful permutation~$\boldsymbol{\alpha}$ exists whenever~$1 \leq \gamma_{\ell} \leq t$~\cite{AK90}.

If~$r-d$ is odd and $d < (r-4)/2$ then, in the notation of the construction, set $\ell=d+1$ and $t = (r-d-1)/2$.  Use the Walecki construction $(1,\ell,2,\ell-1,\ldots, \lfloor (\ell+2)/2 \rfloor)$ for~$\boldsymbol{\gamma}$, giving~$\gamma_{\ell} = \lfloor (\ell+2)/2 \rfloor$.  The construction works provided that~$\gamma_{\ell} \leq t$, which holds when $d < (r-4)/2$.  The absolute  differences of the resulting graceful permutation begin with the absolute differences of~$\boldsymbol{\gamma}$ in order, so the first is~$\ell-1 = d$.

If~$r-d$ is even and~$d \leq (r-8)/2$ then set~$\ell = d+2$ and~$t=(r-d-2)/2$.  Use the 3-twizzler terraces as~$\boldsymbol{\gamma}$; these start~$(2,\ell,1,\ldots)$ and give~$\gamma_{\ell} \in \{ (\ell-1)/2, \ell/2, (\ell+3)/2, (\ell+4)/2\}$ according to the value of~$\ell \imod 6$.  We require~$\gamma_{\ell} \leq t$ so take~$t \geq (\ell+4)/2$; that is, $d \leq (r-8)/2$.   The first absolute difference is~$\ell-2 = d$.

The next case is~$d \geq \lceil r/3 \rceil + 1$.  The imperfect twizzler construction $\boldsymbol{\tau}_{d,r}$ is a graceful permutation whose sequence of absolute differences begins~$(d, d+1, \ldots, r-1, d-1, \ldots  )$.   In particular, the first absolute difference is~$d$.

For~$r<26$ there are 33 cases not covered.  Three are the ones given in the statement of the theorem---an exhaustive search demonstrates that these are not possible.  

There are seven cases with~$d=1$.  Taking the reverse of the Walecki construction covers these.  Reversing the 3-twizzler terraces covers~$(d,r) \in \{ (2,6), (3,7), (3,10), (3,13) \}$.

When~$r = 3d$, here is a new construction of a graceful permutation.  We use semi-colons to help parse the pattern.
$$(d+1; 1, 3d, 2, 3d-1, 3, 3d-2, \ldots, d, 2d+1; d+2, 2d, d+3, 2d-1, \ldots \lceil (3d+1)/2 \rceil).$$
It is straightforward to check that the sequence of absolute differences is
$$(d; 3d-1, 3d-2, \ldots, d+1; d-1, d-2, \ldots 1).$$
This construction covers~$(d,r) \in \{ (3,9), (4,12), (5,15), (6,18), (7,21)\}$.

Examples for the remaining 13 cases are given in Table~\ref{tab:glc}, constructed along similar lines to the general construction of the previous paragraph.

\begin{table}
\caption{Graceful permutations for the proof of Lemma~\ref{lem:diff1}}\label{tab:glc}
$$
\begin{array}{rl}
\hline
(d,r) & {\rm graceful \ permutation} \\
\hline
(2,7) & (4, 6, 1, 7, 3, 2, 5)  \\
(2,10) & (6, 8, 1, 10, 2, 7, 4, 5, 9, 3) \\
(3,8) & (4, 1, 8, 2, 7, 3, 5, 6)  \\
(3,11) & (4, 7, 1, 11, 2, 10, 3, 8, 6, 5, 9) \\
(4,10) & (5, 1, 10, 2, 9, 3, 8, 6, 7, 4) \\
(4,11) & (5, 1, 11, 2, 10, 3, 9, 4, 7, 6, 8) \\
(4,14) &  (5, 9, 1, 14, 2, 13, 3, 12, 6, 11, 4, 7, 8, 10) \\
(5,13) &  (6, 1, 13, 2, 12, 3, 11, 4, 10, 7, 5, 9, 8) \\
(5, 14) & (6, 1, 14, 2, 13, 3, 12, 4, 11, 5, 9, 8, 10, 7) \\
(5,17) & (5, 10, 1, 17, 2, 16, 3, 15, 4, 14, 6, 13, 7, 8, 12, 9, 11) \\
(6,16) & ( 7, 1, 16, 2, 15, 3, 14, 4, 13, 5, 12, 8, 10, 9, 6, 11) \\
(7,19) & (8, 1, 19, 2, 18, 3, 17, 4, 16, 5, 15, 6, 14, 9, 13, 7, 10, 12, 11) \\
(8,22) & (9, 1, 22, 2, 21, 3, 20, 4, 19, 5, 18, 6, 17, 7, 16, 10, 15, 8, 12, 13, 11, 14) \\
(9,25) & (10, 1, 25, 2, 24, 3, 23, 4, 22, 5, 21, 6, 20, 7, 19, 8, 18, 11, 15, 9, 17, 12, 13, 16, 14) \\
\hline
\end{array}
$$
\end{table}

\end{proof}

We are now in position to prove the main result of this section:

\begin{thm}\label{th:oddprime} 
Alspach's Conjecture holds in the case when $n=p$ is prime and $k=p-3$. 
\end{thm}

\begin{proof}  Let~$p=2r+1$.  We first consider ordering the elements of~$\Z_p \setminus \{ 0, d, r+1 \}$ where~$1 \leq d < r$ (when considered as integers).

If~$(d,r) \not\in \{ (2,5), (2,8) \}$, there is a graceful permutation of length~$r$ with first absolute difference~$d$, by Lemma~\ref{lem:diff1}.  Let $\boldsymbol{\alpha} = (\alpha_1, \alpha_2, \ldots, \alpha_r)$ be either this graceful permutation or its complement, whichever has~$\alpha_2 - \alpha_1 = d$.  Let~${\bf a}$ be the directed  rotational terrace constructed using~$\boldsymbol{\alpha}$ via Lemma~\ref{lem:fgm} and let ${\bf b} = (b_1, b_2, \ldots, b_{n-1})$ be the associated  rotational sequencing.

We have~$b_1 = d$ and~$b_{n-1} = -r = r+1$  .  Therefore~$d$ and~$r+1$ appear in adjacent positions of the rotational sequencing and Lemma~\ref{lem:xyodd} gives the required ordering of $\Z_p \setminus \{ 0, d, r+1 \}$.

Next, consider the case~$\Z_p \setminus \{ 0, d', r+1 \}$ where~$r+1 < d' \leq 2r$ (when considered as integers).  Let~$d = p-d'$.  Provided~$(d,r) \not\in \{ (2,5), (2,8) \}$, the rotational sequencing~${\bf b}$ has~$b_{n-2} = d'$ and so~$d'$ and~$r+1$ appear in adjacent positions of the rotational sequencing.  Again, Lemma~\ref{lem:xyodd} gives the required ordering of~$\Z_p \setminus \{ 0, d', r+1 \}$.

Finally, we show that the problem with arbitrary~$x$ and~$y$ removed from~$\Z_p \setminus \{ 0 \}$ can be reduced to one of the above two cases.  Assume that $x \neq \pm y$ (if~$x=-y$ then the sum of the elements in~$\Z_p \setminus \{ 0, x,y \}$ is 0).

Automorphisms of~$\Z_p$ are exactly the multiplications by a nonzero element.  Let~$\nu$ be the element such that~$x\nu = r+1$.  Then~$y\nu \not\in \{0, r,r+1\}$.   The ordering above has~$r+1$ and~$y\nu = \pm d$ missing.  Mutiplying by~$\nu^{-1}$ gives an ordering with~$x$ and~$y$ missing.  This does not cover the cases~$r \in \{5,8\}$ and~$y\nu =  \pm 2$.  In these two cases switching the roles of~$x$ and~$y$ (that is, choosing~$y\nu = r+1$ and~$d= \pm x\nu$) is successful.
\end{proof}

When~$k=p-4$, the approach of this section using rotational sequencings constructed via graceful permutations can certainly handle some instances of the problem.  However, it seems unlikely that a complete solution for~$k=p-4$ is in reach without additional tools.

\section{Concluding remarks}\label{sec:conclude}

Of course the reader may wonder about these conjectures and problem when $n$ is composite.  The polynomial method approach taken in Sections \ref{sec:CN} and \ref{sec:seqs} may be used once again but with limitations.
  
U.~Schauz \cite{S08} has shown that Theorem \ref{th:pm} holds when the field $\F$ is replaced by a ring $R$ so long as no two distinct elements of any $A_i$ differ by a zero-divisor; if so, $A_1 \times \ldots \times A_k$ has what Schauz refers to as {\it Condition D}.  (Of course, Condition D automatically holds for a field.)  

Let $A \subseteq {\mathbb Z}_n$ be a set such that no two elements differ by a zero-divisor and let $p_1$ be the smallest prime dividing $n$.  We claim that $A$ has at most $p_1$ elements and such sets exist.  For the existence, note that the set of integers $\{1,\ldots , p_1\}$ modulo $n$ has pairwise differences of $\{\pm 1, \ldots , \pm(p_1-1)\}$ and none of these is a zero-divisor.  To show that $A$ has at most $p_1$ elements, suppose to the contrary and let $A=\{a_1, \ldots , a_{p_1+1}\}$.  When considered modulo $p_1$, the pigeonhole principle implies that two of these integers belong to the same remainder class and so will differ by a non-zero multiple of $p_1$, which is a zero divisor in ${\mathbb Z}_n$.

Assume $A$ has no two distinct elements that differ by a zero-divisor so that $A^k$ has Condition D.  Thus, we may return to Table \ref{array:alspachcoefficients} and let ${\mathbb N}_{k,j}$ denote the set of non-negative integers greater than 1 that divide $c_{k,j}$.  The coefficient of the monomial $m_{k,j}$ over ${\mathbb Z}_n$ is not zero if and only if $n$ is not in this list.  As before, as $k \leq n-1$ the integers $n$ in this set that do not obey this inequality need not be considered.  This means that Conjecture \ref{conj:alspach} holds for that value of $k$ and all integers $n$ not in this abbreviated list, i.e. for $n > k$ and $n \nmid c_{k,j}$.  We may then turn to a different nonzero coefficient, say, $c_{k,i}$ for $i \neq j$ and repeat this argument. This means that if $\cap_{i=1}^k {\mathbb N}_{k,i}$ is empty or only contains integers less than $k$, then Conjecture \ref{conj:alspach} holds for that value of $k$ and all integers $n$.  For $k \leq 10$, the only instance in which $\cap_{i=1}^k {\mathbb N}_{k,i}$ is non-empty and contains integers greater than $k$ is when $k=8$: when $k=8$ the integer 12 divides $-366468, -92412, 144324$ and $314556$.  However, as shown above, the largest set in ${\mathbb Z}_{12}$ with no two distinct elements differing by a zero-divisor is 2.  So, we reach the conclusion that Conjecture \ref{conj:alspach} is true for sets $A$ with no two distinct elements differing by a zero divisor of size at most $10$ and all $n$.  

Similarly, the constructive methods of Section~\ref{sec:constr} give a partial result for composite~$n$.  The construction method in the proof of Theorem~\ref{th:oddprime} (in conjunction with the computational results for small groups of~\cite{CMPP18}) is sufficient to give an ordering for $S = \Z_n \setminus \{ 0,x,y \}$ for arbitrary~$n$ provided that at least one of~$x$ and~$y$ is coprime to~$n$.

\bibliography{AlspachConjecture-bib.bib}

\end{document}